\documentclass[11pt]{amsart}
\usepackage[a4paper,margin=2.5cm]{geometry}
\usepackage{hyperref,mathtools,amsmath,amsthm,amssymb,amscd,amsfonts,color,tikz-cd}
\usepackage{tikz}
\usetikzlibrary{arrows}

\theoremstyle{plain}
\newtheorem{theorem}{Theorem}[section]

\newtheorem{lemma}[theorem]{Lemma}

\numberwithin{equation}{section}

\theoremstyle{definition}
\newtheorem{definition}[theorem]{Definition}
\newtheorem{remark}[theorem]{Remark}
\newtheorem{example}[theorem]{Example}

\newtheorem{conjecture}[theorem]{Conjecture}

\newcommand{\cP}{\mathcal{P}}

\newcommand{\cC}{\mathcal{C}}

\newcommand{\cI}{\mathcal{I}}

\newcommand{\cR}{\mathcal{R}}

\newcommand{\cX}{\mathcal{X}}

\newcommand{\fS}{\mathfrak{S}}

\newcommand{\Q}{\mathbb{Q}}
\newcommand{\R}{\mathbb{R}}
\newcommand{\Z}{\mathbb{Z}}
\newcommand{\N}{\mathbb{N}}

\newcommand{\ie}{\textit{i}.\textit{e}. }

\newcommand{\set}{\Sigma}

\DeclareMathOperator{\Inv}{Inv}
\DeclareMathOperator{\even}{even}
\DeclareMathOperator{\odd}{odd}
\DeclareMathOperator{\Perm}{Perm}
\DeclareMathOperator{\Alt}{Alt}
\DeclareMathOperator{\sgn}{sgn}

\begin{document}
\title[The cohomology rings of real permutohedral varieties]{The cohomology rings of real permutohedral varieties}

\author[S. Choi]{Suyoung Choi}
\address{Department of mathematics, Ajou University, 206, World cup-ro, Yeongtong-gu, Suwon 16499,  Republic of Korea}
\email{schoi@ajou.ac.kr}

\author[Y. Yoon]{Younghan Yoon}
\address{Department of mathematics, Ajou University, 206, World cup-ro, Yeongtong-gu, Suwon 16499,  Republic of Korea}
\email{younghan300@ajou.ac.kr}

\date{\today}

\thanks{The authors were supported by the National Research Foundation of Korea Grant funded by the Korean Government (NRF-2019R1A2C2010989).}


\begin{abstract}
    A permutohedral variety is a remarkable object in various areas of mathematics, and its topological invariants are widely recognized.
    However, only little is known about a real permutohedral variety, that is, the real locus of a permutohedral variety.
    
    The rational Betti numbers of real permutohedral varieties were computed in terms of alternating permutations in 2012.
    In this paper, we provide explicit descriptions of the cohomology ring of real permutohedral varieties.
    In particular, we describe the multiplicative structure in terms of alternating permutations.
    \end{abstract}
\maketitle
\tableofcontents

\section{Introduction}
Let $\Phi_R$ denote an irreducible root system of type $R$ with rank $n$, and let $W_R$ be the associated Weyl group.
A valuable approach to comprehending the root system $\Phi_R$ is to view it as a reflection group within the simplicial complex.
This complex, known as the \emph{Coxeter complex} $K_R$, consists of Weyl chambers of $W_R$ and its dimension is $n-1$.
The set of $0$-simplices of $K_R$ is precisely the $W_R$-orbit of the fundamental co-weights.
It is well-known that the combination of $K_R$ and the co-weights gives rise to a non-singular complete fan.
Through the fundamental theorem for toric geometry, we obtain the corresponding smooth compact toric variety $X_R$, which is commonly known as the \emph{Coxeter toric variety} of type $R$.
For further details, refer to \cite{Procesi1990}.
In particular, Coxeter toric variety of type $A$ is also known as the \emph{permutohedral variety}, one of the Hessenberg varieties \cite{DeMari-Procesi-Shayman1992}, and is isomorphic to a specific De Concini-Procesi wonderful model as well \cite{DeConcini-Procesi1995}.
The permutohedral variety has been studied in various areas of mathematics \cite{Klyachko1985}, \cite{Mari1988}, \cite{Procesi1990}, \cite{JuneHuh2012}, and \cite{Huh2014PhD}.

For each toric variety $X$, there is a canonical involution induced from a complex conjugation.
The real locus $X^\R$ of $X$, that is, the fixed point set of $X$, is a real algebraic variety, and it is called a \emph{real toric variety}.
In particular, the real locus $X^{\R}_R$ of a Coxeter toric variety $X_R$ is known as the \emph{real Coxeter toric variety} of type $R$, with a dimension of $n$.
Generally, computing the topological invariants of real algebraic varieties is a highly intricate task.
However, the real Coxeter variety exhibits a well-structured nature, prompting numerous attempts to compute its topological invariants.
Remarkably, the computation of the rational Betti numbers of $X^{\R}_R$ have been completely done for each Weyl group \cite{Henderson2012},\cite{Choi-Park-Park2017typeB},\cite{Choi-Kaji-Park2019},\cite{Cho-Choi-Kaji2019}, and \cite{Choi-Yoon-Yu2024}.

Nevertheless, the multiplicative structure of $H^\ast(X^\R_R; \Q)$ remains largely unexplored.
The primary aim of this paper is to provide an explicit description of the cohomology ring for \emph{real permutohedral varieties}, that is, real Coxeter toric varieties $X^\R_A$ of type~$A$.
As previously mentioned, the additive structure was already computed by Henderson \cite{Henderson2012} and later recomputed by Suciu \cite{Suciu2012} using entirely different approaches.
It is worth noting that the $k$th $\Q$-Betti number~$\beta^k$ of the real permutohedral varieties correspond to certain combinatorial integer sequences
\begin{equation}\label{typeA_betti}
\beta^k(X^{\R}_{A_n}) = \binom{n+1}{2k} a_{2k},
\end{equation}
where $a_{2k}$ represents the $2k$th Euler zigzag number (A000111 in \cite{oeis}), which corresponds to the number of alternating permutations of length~$2k$.
Therefore, each generator of $\beta^k(X^{\R}_{A_n})$ can be associated with an alternating permutation of length~$2k$ on $I \in \binom{[n+1]}{2k}$.

Denote by $\Perm_I$ the set of permutations on $I$, and by $\Alt_I$ the set of alternating permutations on $I$.
In Section~\ref{section : type $A$}, we firstly introduce explicit $\Q$-vector space isomorphisms
$$
    \bigoplus_{I \in \binom{[n+1]}{2k}}\Q\left\langle \Alt_I \right\rangle  \longrightarrow \bigoplus_{I \in \binom{[n+1]}{2k}} \frac{\Q\left\langle \Perm_I  \right\rangle}{M_I} \longrightarrow H_{k}(X^{\R}_{A_n})  \longrightarrow H^{k}(X^{\R}_{A_n}),
$$
where $M_I$ is a submodule of $\Q\left\langle \Perm_I \right\rangle$ generated by
\begin{enumerate}
    \item $(x_1\cdots /x_{2i-1}x_{2i}/\cdots x_{2k}) + (x_1\cdots /x_{2i}x_{2i-1}/\cdots x_{2k})$ for $1 \leq i \leq k-1$,
    \item $(x_1\cdots/x_{2i-1}x_{2i}/x_{2i+1}x_{2i+2}/\cdots x_{2k}) -(x_1\cdots/x_{2i-1}x_{2i+1}/x_{2i}x_{2i+2}/\cdots x_{2k})$\\ $+(x_1\cdots/x_{2i-1}x_{2i+2}/x_{2i}x_{2i+1}/\cdots x_{2k}) +(x_1\cdots/x_{2i}x_{2i+1}/x_{2i-1}x_{2i+2}/\cdots x_{2k})$\\ $-(x_1\cdots/x_{2i}x_{2i+2}/x_{2i-1}x_{2i+1}/\cdots x_{2k}) +(x_1\cdots/x_{2i+1}x_{2i+2}/x_{2i-1}x_{2i}/\cdots x_{2k})$\\ for $1 \leq i \leq k-1$.
\end{enumerate}
The above relations are well known as the \emph{Garnir relation} for Young tableaux \cite{Sagan2001book}.
For convenience, every permutation (respectively, alternating permutation) $\alpha \in \Perm_I$ is regarded as a vector of $\Q\left\langle \Perm_I \right\rangle$ (respectively, $\Q\left\langle \Alt_I \right\rangle$).
By the isomorphisms above, $x \in \Perm_I$ can be uniquely written by a linear combination of $\Alt_I$ over ${\Q\left\langle \Perm_I  \right\rangle}/{M_I}$ as
$$
x = \sum_{\alpha \in \Alt_I} \cC^{\alpha}_x \alpha
$$
for some~$\cC^{\alpha}_x \in \Q$.

In Section~\ref{section : multi}, we furnish explicit descriptions of the multiplicative structure for $H^\ast(X^\R_{A};\Q)$ of~$X^\R_{A}$, utilizing alternating permutations.
To present the main result, we briefly introduce some notations :

For $I_1 \in \binom{[n+1]}{2k}$ and $I_2 \in \binom{[n+1]}{2l}$ with $I_1 \cap I_2 =\emptyset$,
\begin{itemize}
  \item $\cR_{I_1,I_2} = \{z = (z_1 \cdots z_{2(k+l)}) \in \Alt_{I_1 \cup I_2} \mid \text{$\{z_{2i-1},z_{2i}\} \subset I_1$ or $\{z_{2i-1},z_{2i}\} \subset I_2$ for all $i$}\}$,
  \item $\rho_{I_1}(z) \in \Perm_{I_1}$ and $\rho_{I_2}(z) \in \Perm_{I_2}$ are subpermutations of $z \in \cR_{I_1,I_2}$, and
  \item for $z = (z_1\cdots z_{2(k+l)}) \in \cR_{I_1,I_2}$, $\kappa_{I_1,I_2}(z)$ is the sign of $\sigma \in \fS_{k+l}$ with
  $$
  \begin{cases}
\rho_{I_1}(z) = (z_{2\sigma(1)-1}z_{2\sigma(1)}/\cdots/z_{2\sigma(k)-1}z_{2\sigma(k)}), \mbox{ and}\\
\rho_{I_2}(z) = (z_{2\sigma(k+1)-1}z_{2\sigma(k+1)}/\cdots/z_{2\sigma(k+l)-1}z_{2\sigma(k+l)}).
\end{cases}$$
\end{itemize}

\begin{theorem}
For any $n \in \N$, there is a $\Q$-algebra isomorphism
$$
\bigoplus_{k=0}^{\lfloor \frac{n+1}{2} \rfloor} \bigoplus_{I \in \binom{[n+1]}{2k}} \frac{\Q\left\langle \Perm_I  \right\rangle}{M_{I}} \cong H^\ast(X^\R_{A_n};\Q),
$$
where the multiplicative structure $\smile \colon \frac{\Q\left\langle \Perm_{I_1}  \right\rangle}{M_{I_1}} \otimes \frac{\Q\left\langle \Perm_{I_2}  \right\rangle}{M_{I_2}} \to \frac{\Q\left\langle \Perm_{I_1 \triangle I_2}  \right\rangle}{M_{I_1 \triangle I_2}}$ is defined by, for $\alpha \in \Alt_{I_1}$ and $\beta \in \Alt_{I_2}$,
$$
\alpha \smile \beta = \begin{cases}
\underset{z \in \cR_{I_1,I_2}}{\sum}(-1)^{\kappa_{I_1,I_2}(z)} \cC^\alpha_{\rho_{I_1}(z)} \cC^\beta_{\rho_{I_2}(z)} \cdot z, & \mbox{if } I_1 \cap I_2 = \emptyset,\\
0, & \mbox{otherwise.}
\end{cases}
$$
\end{theorem}

In Section~\ref{section : Coefficients}, we introduce some notions of \emph{incidence algebras}, as introduced in \cite{Rota1964}.
Each coefficients $\cC^\alpha_x$ can be regarded as values of an integral function of the incidence algebra of $\Perm_I$ with some specific partial order.

\section{Preliminaries}

Let $X$ be an $n$-dimensional smooth compact (complex) toric variety.
By the fundamental theorem for toric geometry, $X$ is completely determined by a non-singular complete fan $\Sigma \in \R^n$.
From the set $V$ of rays of $\Sigma$, we obtain
\begin{enumerate}
  \item a simplicial complex $K$ on $V$, called the \emph{underlying simplicial complex} of $X$, whose simplices correspond to cones of $\Sigma$, and
  \item a linear map $\lambda \colon V \to \Z^n$, called the \emph{characteristic map} of $X$, such that $\lambda(v)$ is the primitive vector in the direction of $v$ for each $v \in V$.
\end{enumerate}
Since $\Sigma$ can be regarded as the pair $(K,\lambda)$, $X$ is completely determined by $(K,\lambda)$.

Similarly, a real toric variety $X^{\R}$, which is the real locus of $X$, is completely determined by a pair $(K,\lambda^{\R})$, where the (mod $2$) characteristic map $\lambda^{\R} \colon V \xrightarrow{\lambda} \Z^n \xrightarrow{\text{mod} 2} \Z_2^n$.
We may set that $\vert V \vert =m$.
Then, each $\lambda^{\R}$ can be expressed by an $n \times m$ (mod $2$) matrix $\Lambda$ whose each column corresponds to the image of $\lambda^\R$ of the respective ray.
The matrix $\Lambda$ is called the (mod $2$) \emph{characteristic matrix} of $X^{\R}$.

Denote by $\text{row}(\Lambda)$ the row space of $\Lambda$.
Then, $\text{row}(\Lambda)$ is a subspace of $\Z_2^m$, and each element $S \in \text{row}(\Lambda)$ correspond to the induced subcomplex $K_S$ of $K$ with respect to $S$, where $S$ is identified with a subset of $V$ through the standard correlation between the power set of $V$ and  $\Z_2^m$; in this arrangement, each element of $V$ corresponds to a coordinate of $\Z_2^m$.

Note that computing the (co)homology of $X^{\R}$ depends only on $K$ and $\Lambda$, based the following theorem.

\begin{theorem}\cite{Choi-Park2017_multiplicative}\label{Choi-Park2017 multimplicative}
There is a $(\Z \oplus \text{row}(\Lambda))$-graded $\Q$-algebra isomorphism
$$
    H^{\ast}(X^{\R}(K),\Q) \cong \underset{S \in \text{row}(\Lambda)}\bigoplus \widetilde{H}^{\ast-1}(K_S,\Q).
$$
In particular, the product structure on $\underset{S \in \text{row}(\Lambda)}\bigoplus \widetilde{H}^{\ast}(K_S)$ is given by the canonical maps
$$
    \widetilde{H}^{k-1}(K_{S_1}) \otimes \widetilde{H}^{l-1}(K_{S_2}) \to \widetilde{H}^{k+l-1}(K_{S_1+S_2})
$$
which are induced by simplicial maps $K_{S_1+S_2} \to K_{S_1} \star K_{S_2}$, where $\star$ denotes the simplicial join.
\end{theorem}

For any irreducible root system of type $R$ in a finite dimensional Euclidean space $E$, its Weyl group gives connected components in $E$, called the \emph{Weyl chambers}.
Consider the fan $\Sigma_R$, whose maximal cones are Weyl chambers of type $R$.
Note the fact that $\Sigma_R$ is a non-singular complete fan for any irreducible root system of type $R$ \cite{Procesi1990}.
Hence, each irreducible root system of type $R$ gives the complex and real toric varieties, denoted by $X_R$ and $X_R^{\R}$, respectively.
In addition, $K_R$ denotes a underlying simplicial complex or $X_R$, $\lambda^{\R}_R$ denotes a (mod $2$) characteristic map of $X^{\R}_R$, and $\Lambda_R$ denotes a (mod $2$) characteristic matrix of $X_R^{\R}$.

In the rest of this section, we review the construction of $(K_R,\Lambda_R)$, where type $R$ is $A_n$.
As is shown in \cite{Suciu2012} (or \cite{Cho-Choi-Kaji2019}), each vertex of $K_{A_n}$ can be identified by the nonempty proper subsets of $[n+1] = \{1, \ldots, n+1\}$, and each $k$-simplex of $K_{A_n}$ can be regarded as nested $k+1$ nonempty proper subsets of $[n+1]$.
Also, the (mod $2$) characteristic map $\lambda^{\R}_{A_n}$ is defined by $\lambda^{\R}_{A_n}(L) = \sum_{k \in L}\mathbf{e}_k$ for vertex $L$ in $K_{A_n}$, where $\mathbf{e}_k$ is the $k$th standard vector of $\Z_2^n$ for $1 \leq k \leq n$, and $\mathbf{e}_{n+1} = \sum_{i=1}^{n}\mathbf{e}_i$.
Every element $S$ of $\text{row}(\Lambda_{A_n})$ is a sum of rows, say $i_1,i_2,\ldots,i_s$th rows.
Then, $S$ is identified with a subset $I_S$ of $[n+1] = \{1, \ldots, n+1\}$ whose cardinality is even in the following way:
$$
I_S=
\begin{cases}
\{i_1,i_2,\ldots,i_s\}, & \mbox{if }s\mbox{ is even} \\
\{i_1,i_2,\ldots,i_s,n+1\}, & \mbox{if }s\mbox{ is odd}
\end{cases}
$$
For $S \in \text{row}(\Lambda_{A_n})$, $(K_{A_n})_{I_S}$ is the full-subcomplex of $K_{A_n}$ consisting of the vertices $J$ such that $\vert J \cap I_S \vert$ is odd.
Also, $I_{S_1+S_2} = I_{S_1} \triangle I_{S_2}$ for $S_1$ and $S_2 \in \text{row}(\Lambda_{A_n})$, where $\triangle$ is the symmetric difference defined by $I_1 \triangle I_2 = (I_1 \cup I_2)\setminus(I_1 \cap I_2).$
Thus, throughout this paper, $I \in \binom{[n+1]}{2k}$ is regarded as an element of the row space of $\Lambda_{A_n}$.

\begin{example}
    Consider $A_7$ and $I = \{1,3,4,5\}$. Then, the set of vertices of $(K_{A_7})_I$ consists of all unions of an odd subset of $\{1,3,4,5\}$ and a subset of $\{2,6,7\}$.
    Thus, the number of vertices of $(K_{A_7})_I$ is $(\binom{4}{1} + \binom{4}{3}) \times 2^3 = 64$.
\end{example}

For $(K_{A_{n}})_{I}$, we define $\widehat{(K_{A_{n}})_{I}}$ as the subcomplex of $(K_{A_{n}})_{I}$ by iteratively removing each vertex that is not a subset of $I$ from $(K_{A_{n}})_{I}$.
From the proof of Lemma~5.3 in \cite{Choi-Park2015}, we conclude the following lemma.

\begin{lemma}\label{typeA_homotopy}
    Let $\pi \colon (K_{A_{n}})_{I} \to \widehat{(K_{A_{n}})_{I}}$ be the simplicial map defined by $\pi(\sigma) = \sigma \cap I$ for all $\sigma \in (K_{A_{n}})_{I}$.
    Then, $\pi$ is a homotopy equivalence.
\end{lemma}

Thus, for each (co)homology of $(K_{A_n})_I$, it is unnecessary to consider outside of the subcomplex $\widehat{(K_{A_{n}})_{I}}$ of $(K_{A_n})_I$.
\begin{example}
Consider $A_7$ and $I = \{1,3,4,5\}$. By Lemma~\ref{typeA_homotopy}, since $7 \notin I$,
(a) and (b) are the same in $H^{1}((K_{A_7})_I)$ as following :
  \begin{center}
        \begin{tikzpicture}[scale = 1]
            \coordinate [label = above left:{\tiny $\{1\}$}] (1) at (-4,1);
            \coordinate [label = below right:{\tiny $\{3\}$}] (3) at (-2,-1);
            \coordinate [label = above right:{\tiny$\{1,3,4,7\}$}] (1347) at (4,1);
            \coordinate [label = below left:{\tiny$\{1,3,5,7\}$}] (1357) at (2,-1);
            \coordinate [label = above right:{\tiny $\{1,3,4\}$}] (134) at (-2,1);
            \coordinate [label = below left:{\tiny $\{1,3,5\}$}] (135) at (-4,-1);
            \coordinate [label = below right:{\tiny $\{3,7\}$}] (37) at (4,-1);
            \coordinate [label = above left:{\tiny $\{1,7\}$}] (17) at (2,1);
            \path [thick, ->] (1) edge (134) (134) edge (3) (3) edge (135) (135) edge (1);
            \draw (-3,0) node {(a)};
            \path [thick, ->] (17) edge (1347) (1347) edge (37) (37) edge (1357) (1357) edge (17);
            \draw (3,0) node {(b)};
            \draw (0,0) node {$=$};
        \end{tikzpicture}
     \end{center}
\end{example}

By \cite{Choi-Park2015}, $(K_{A_n})_I$ is a shellable complex, and, hence, it is Cohen Macaulay \cite{Stanley1996book}.
Therefore, it is homotopy equivalent to a wedge sum of spheres of the same dimension.
By calculating the Euler numbers, one can compute
\begin{equation}\label{typeA_I}
\beta_{k-1}((K_{A_n})_I) = a_{2k},
\end{equation}
for any $I \in \binom{[n+1]}{2k}$.
This, in turn, implies \eqref{typeA_betti}.

\begin{remark}
    According to \cite{Cai-Choi2021}, the integral cohomology of $X^\R$ corresponding to $(K,\Lambda)$ can be computed if $K$ is shellable.
    The $\Z_2$-Betti numbers align with the $h$-numbers of $K$, and other torsions are determined by the torsion of $K_S$ for $S \in \text{row}(\Lambda)$.
    In the case of permutohedral varieties, as $K_S$ has no torsion, $X^\R_{A_n}$ only exhibits $\Z_2$-torsions, which provides the integral cohomology group of $X^\R_{A_n}$.
\end{remark}

\section{Permutations as elements in $H^\ast(X^\R_{A_n})$}\label{section : type $A$}

Let $I \in \binom{[n+1]}{2k}$ be a subset of $[n+1] = \{1, \ldots, n+1\}$ whose cardinality is $2k$.
Let us consider the set $\Perm_I$ of permutations on $I$.
For each $x \in \Perm_I$, we use the one-line notation $(x_1x_2/\cdots/x_{2k-1}x_{2k})$ to represent $x$ which is grouped into blocks of length two, separated by a slash, where $x^{-1}(x_1) < x^{-1}(x_2) < \cdots < x^{-1}(x_{2k})$.
For each permutation $x=(x_1x_2/\cdots/x_{2k-1}x_{2k}) \in \Perm_I$ and $1 \leq i \leq 2k$, we introduce the notation
$$
    \set^x_i = \left\{
                 \begin{array}{ll}
                   \{x_i\}, & \hbox{if $i = 1, 2$;} \\
                   \{x_1, x_2, \ldots, x_{2\lfloor \frac{i-1}{2} \rfloor}, x_i \}, & \hbox{if $i \geq 3$,}
                 \end{array}
               \right.
$$
and we associate to $x$ a subcomplex $\Delta_x$ of $(K_{A_n})_I$ as
$$\Delta_x = \{ \set^x_1 , \set^x_2\} \star \{\set^x_3,\set^x_4\} \star \cdots \star \{\set^x_{2k-1},\set^x_{2k}\},$$
where $\star$ is the simplicial join.
Note that $\Delta_x$ is the boundary complex of the cross-polytope of dimension~$k$ with fixed orientation embedded in $(K_{A_n})_I$.

\begin{example}
    Let $x = (x_1x_2)$ be a permutation of length~$2$.
    Then $\Delta_{(12)}$ is $\{\{x_1\}, \{x_2\}\}$, and is homeomorphic to the $0$-dimensional sphere $S^0$.
    Let $x=(x_1 x_2 /x_3 x_4)$ be a permutation of length~$4$.
    Then, $\Delta_{(x_1 x_2 /x_3 x_4)}$ (with the orientation), homeomorphic to the $1$-dimensional sphere $S^1$, is as follows.
     \begin{center}
        \begin{tikzpicture}[scale = 1]
            \coordinate [label = above left:{\tiny $\set^x_1 = \{x_{1}\}$}] (1) at (-1,1);
            \coordinate [label = below right:{\tiny $\set^x_2 = \{x_{2}\}$}] (2) at (1,-1);
            \coordinate [label = above right:{\tiny $\set^x_3 = \{x_{1},x_2,x_3\}$}] (123) at (1,1);
            \coordinate [label = below left:{\tiny $\set^x_4 = \{x_1,x_2,x_4\}$}] (124) at (-1,-1);
            \path [thick, ->] (1) edge (123) (123) edge (2) (2) edge (124) (124) edge (1);
            \draw (0,0) node {$\Delta_{(x_1 x_2 /x_3 x_4)}$};
        \end{tikzpicture}
     \end{center}
\end{example}

Let $\Phi_I \colon \Q\left\langle \Perm_I \right\rangle \to \widetilde{H}_{k-1}((K_{A_n})_I)$ be a linear map defined by
$$
    \Phi_I(x) = [\Delta_x].
$$

\begin{example}\label{A7}
	Consider the type $A_7$. Let $I = \{1,3,4,5\} \in \binom{[8]}{4}$. Then,
    $\widehat{(K_{A_7})}_{I}$ is
     \begin{center}
        \begin{tikzpicture}[scale = 1]
            \coordinate [label = below right:{\small $\{1\}$}] (1) at (-1,1);
            \coordinate [label = above left:{\small $\{3\}$}] (2) at (1,-1);
            \coordinate [label = above right:{\small$\{4\}$}] (3) at (3,2);
            \coordinate [label = below left:{\small$\{5\}$}] (4) at (-3,-2);
            \coordinate [label = below right:{\tiny $\{1,3,4\}$}] (123) at (1,1);
            \coordinate [label = above left:{\tiny $\{1,3,5\}$}] (124) at (-1,-1);
            \coordinate [label = below:{\tiny $\{3,4,5\}$}] (234) at (3,-2);
            \coordinate [label = above:{\tiny $\{1,4,5\}$}] (134) at (-3,2);
            \draw [ultra thick, miter limit=2] (1)--(123)--(2)--(124)--cycle;
            \draw [ultra thick, miter limit=2] (3)--(134)--(4)--(234)--cycle;
            \draw [ultra thick, miter limit=2] (1)--(134);
            \draw [ultra thick, miter limit=2] (2)--(234);
            \draw [ultra thick, miter limit=2] (3)--(123);
            \draw [ultra thick, miter limit=2] (4)--(124);
            \draw (0,0) node {$\Delta_{(13 /45)}$};
            \draw (0,1.5) node {$\Delta_{(14 /35)}$};
            \draw (0,-1.5) node {$\Delta_{(35 /14)}$};
            \draw (-2,0) node {$\Delta_{(15 /34)}$};
            \draw (2,0) node {$\Delta_{(34 /15)}$};
        \end{tikzpicture},
     \end{center}
     where the outer square-shape complex is $\Delta_{(45/ 13)}$.

     In particular,
\begin{align*}
    \Phi_I((15/34)) &= [\{ \set^{(15/34)}_1 , \set^{\tiny (15/34)}_2\} \star \{\set^{\tiny (15/34)}_3,\set^{\tiny (15/34)}_4\}] \\
    &= [\{\{1\},\{5\}\}\star\{\{1,5,3\},\{1,5,4\}\}] \\
    &= [\{\{1\},\{1,5,3\}\}-\{\{1\},\{1,5,4\}\}-\{\{5\},\{1,5,3\}\}+\{\{5\},\{1,5,4\}\}].
\end{align*}
 \end{example}

\begin{remark} \label{rem:relations}
We note that $\Phi_I(x)$ is the homology class with respect to $\Delta_x$.
For a simplicity, if we set $x = (\bar{x}/ x_{2i-1} x_{2i} / \bar{\bar{x}})$, then $\Delta_{(\bar{x}/ x_{2i} x_{2i-1}/ \bar{\bar{x}})}$ is the same complex to $\Delta_x$ with reversed orientation.
In addition, if we set $x = (\bar{x}/ x_{2i-1} x_{2i} / x_{2i+1} x_{2i+2} / \bar{\bar{x}})$, we have $(K_{A_n})_I$ contains the subcomplex $\Delta_{\bar{x}} \star (K_{A_n})_{\{x_{2i-1}, x_{2i},x_{2i+1},x_{2i+2}\}} \star \Delta_{\bar{\bar{x}}}$.
Therefore, the topological structure of $(K_{A_n})_I$ gives the relations between the image of $\Phi_I$ as following:

\begin{enumerate}
    \item $\Phi_I(x_1\cdots /x_{2i-1}x_{2i}/\cdots x_{2k}) = - \Phi_I(x_1\cdots /x_{2i}x_{2i-1}/\cdots x_{2k})$
    \item $\Phi_I((x_1\cdots/x_{2i-1}x_{2i}/x_{2i+1}x_{2i+2}/\cdots x_{2k}))=\Phi_I((x_1\cdots/x_{2i-1}x_{2i+1}/x_{2i}x_{2i+2}/\cdots x_{2k}))$\\
        $-\Phi_I((x_1\cdots/x_{2i-1}x_{2i+2}/x_{2i}x_{2i+1}/\cdots x_{2k}))-\Phi_I((x_1\cdots/x_{2i}x_{2i+1}/x_{2i-1}x_{2i+2}/\cdots x_{2k}))$\\
        $+\Phi_I((x_1\cdots/x_{2i}x_{2i+2}/x_{2i-1}x_{2i+1}/\cdots x_{2k}))-\Phi_I((x_1\cdots/x_{2i+1}x_{2i+2}/x_{2i-1}x_{2i}/\cdots x_{2k}))$\\
        for $1 \leq i \leq k-1$.
\end{enumerate}
\end{remark}

Note that the permutation group $\fS_{2k}$ on the index set $\{1,\ldots,2k\}$ acts on the set~$\Perm_I$ by
$$\sigma \cdot x = (x_{\sigma(1)}x_{\sigma(2)}/\cdots/x_{\sigma(2k-1)}x_{\sigma(2k)})$$
for each $x=(x_1x_2/\cdots/x_{2k-1}x_{2k}) \in \Perm_I$ and $\sigma \in \fS_{2k}$.
For convenience, we define subgroups of $\fS_{2k}$
\begin{enumerate}
  \item $\fS_{2k}^{\text{even}} =\{\sigma_1\sigma_2\cdots\sigma_{k} \mid \sigma_i \in \fS_{\{2i-1,2i\}} \text{ for } 1\leq i \leq k\}$, and
  \item $\fS_{2k}^{\text{odd}} =\{\sigma_1\sigma_2\cdots\sigma_{k-1}\mid \sigma_i \in \fS_{\{2i,2i+1\}} \text{ for } 1\leq i \leq k-1\}$.
\end{enumerate}

Let $\varphi_I : \Perm_I \to (K_{A_n})_{I}$ defined by
$$
\varphi_I(x) = \{\set^x_1,\set^x_3,\ldots,\set^x_{2k-1}\}
$$
for $x \in \Perm_I$.

\begin{lemma}\label{face_lemma}
Let $x$ and $y$ be permutations on $I$.
Then, $\varphi_I(x)$ is a face of $\Delta_y$ if and only if there are $\tau_{x,y} \in \fS_{2k}^{\text{even}}$ and $\sigma_{x,y} \in \fS_{2k}^{\text{odd}}$ such that $\tau_{x,y} \cdot \sigma_{x,y} \cdot x = y$.
Furthermore, $\tau_{x,y}$ and $\sigma_{x,y}$ are uniquely determined for $x$ and $y$.
\end{lemma}
\begin{proof}
    The `if' part is clear. For the opposite direction, assume that the simplex $\varphi_I(x)$ is a face of $\Delta_y$.
    Then, $x_{1}$ (respectively, $x_{2k}$) must be a member of $\{y_{1},y_{2}\}$ (respectively, $\{y_{2k-1},y_{2k}\}$).
    For each $1 \leq i < k$, $\alpha_{2i}$ and $\alpha_{2i+1}$ must fulfill exactly one of the following conditions:
    \begin{enumerate}
      \item $x_{2i} \in \{y_{2i-1},y_{2i}\}, x_{2i+1} \in \{y_{2i+1},y_{2i+2}\}$, or
      \item $x_{2i} \in \{y_{2i+1},y_{2i+2}\}, x_{2i+1} \in \{y_{2i-1},y_{2i}\}$.
    \end{enumerate}
    Thus there uniquely exists $\sigma_{x,y} \in \fS_{2k}^{\text{odd}}$ such that
    $$
    \{y_{2i-1},y_{2i}\} = \{(\sigma_{x,y}\cdot x)_{2i-1},(\sigma_{x,y}\cdot x)_{2i}\}.
    $$
    Moreover, there uniquely exists $\tau_{x,y} \in \fS_{2k}^{\text{even}}$ such that $\tau_{x,y} \cdot (\sigma_{x,y} \cdot x) = y$ as desired.
\end{proof}

An \emph{alternating permutation}, also known as a \emph{zig-zag permutation}, on $I$ is defined as a permutation within $\Perm_I$ that presents an oscillating pattern where each entry is either larger or smaller than the one before it in an alternating manner.
We will refer to the collection of these permutations on $I$ as $\Alt_I$.
Note that $\varphi_I \vert_{\Alt_I} \colon \Alt_I \to (K_{A_n})_I$ is injective.

It is worth noting that the rank of $\widetilde{H}_{k-1}((K_{A_{n}})_{I})$ is $a_{2k}$ as seen in \eqref{typeA_I}, and $a_{2k}$ represents the cardinality of $\Alt_I$.
We firstly aim to demonstrate that the image under the mapping $\Phi_I$, denoted by $\Phi_I (\Alt_I)$, indeed forms a $\Q$-basis for the space $\widetilde{H}_{k-1}((K_{A_{n}})_{I})$.

We now introduce the \emph{right weak order} $\leq_R$ on $\Perm_I$, as detailed in \cite{Bjorner_book2005}. 
For a given permutation $x \in \Perm_I$, the set of inversions, denoted by $\Inv(x)$, is defined as
$$
\Inv(x) = \{(x_i,x_j) \in I \times I \mid 1 \leq i < j \leq 2k \text{ and } x_i > x_j \}.
$$
Given permutations $x$ and $y$ on $I$, $x \leq_R y$ if $\Inv(x) \subseteq \Inv(y)$.
By Lemma~\ref{face_lemma}, one can directly show that
\begin{equation}\label{min_element}
  \alpha \in \Alt_I \text{ is a maximal element of } (\{x \in \Perm_I \mid \varphi_I(\alpha) \text{ is a face of } \Delta_x \},\leq_R).
\end{equation}
Consequently, we conclude that
$\{\beta \in \Alt_I \mid \varphi_I(\alpha) \text{ is a face of } \Delta_\beta\} = \{\alpha\}$
for each minimal element $\alpha$ of $(\Alt_I,\leq_R\vert_{\Alt_I})$.

\begin{lemma}\label{lemma:basis_A}
    The image $\Phi_I(\Alt_I) = \{\Phi_{I}(x) \mid x \in \Alt_I\}$ is a $\Q$-basis of $\widetilde{H}_{k-1}((K_{A_{n}})_{I})$.
\end{lemma}

\begin{proof}
By Lemma~\ref{typeA_homotopy}, it is sufficient to show that $\{\Phi_{I}(x) \mid x \in \Alt_I\}$ is a $\Q$-basis of $\widetilde{H}_{k-1}(\widehat{(K_{A_{n}})_{I}})$.
Since the dimension~of $\widehat{(K_{A_{n}})_{I}}$ is $k-1$, every element of $\widetilde{H}_{k-1}(\widehat{(K_{A_{n}})_{I}})$ can be regarded as a cycle of  $C_{k-1}((\widehat{K_{A_{n}})_{I}})$. Suppose that a $\Q$-linear combination $\cX$ of $\{\Delta_\alpha \mid \alpha \in \Alt_I\}$ becomes zero.
Let $\mathfrak{m}$ be the set of minimal elements of $(\Alt_I,\leq_R\vert_{\Alt_I})$.
For each elements $\alpha$ and $\beta$ in~$\mathfrak{m}$, a simplex $\varphi_I(\alpha)$ is contained in $\Delta_\beta$ if and only if $\alpha = \beta$.
Thus, the coefficient of $\alpha$ in $\cX$ is zero for all $\alpha \in \mathfrak{m}$.
By similar way, for all minimal elements $\alpha$ of
$$
(\Alt_I \setminus \mathfrak{m} , \leq_R\vert_{\Alt_I \setminus \mathfrak{m}}),
$$
the coefficient of $\alpha$ in $\cX$ is zero.
By repeating this process, it can be shown that all coefficients of $\cX$ are zero.
Therefore, we have $\Phi_I(\Alt_I)$ is $\Q$-linearly independent.
\end{proof}

We define a submodule $M_I$ of $\Q\left\langle \Perm_I \right\rangle$ generated by

\begin{enumerate}
    \item $(x_1\cdots /x_{2i-1}x_{2i}/\cdots x_{2k}) + (x_1\cdots /x_{2i}x_{2i-1}/\cdots x_{2k})$ for $1 \leq i \leq k$,
    \item $(x_1\cdots/x_{2i-1}x_{2i}/x_{2i+1}x_{2i+2}/\cdots x_{2k}) -(x_1\cdots/x_{2i-1}x_{2i+1}/x_{2i}x_{2i+2}/\cdots x_{2k})$\\ $+(x_1\cdots/x_{2i-1}x_{2i+2}/x_{2i}x_{2i+1}/\cdots x_{2k}) +(x_1\cdots/x_{2i}x_{2i+1}/x_{2i-1}x_{2i+2}/\cdots x_{2k})$\\ $-(x_1\cdots/x_{2i}x_{2i+2}/x_{2i-1}x_{2i+1}/\cdots x_{2k}) +(x_1\cdots/x_{2i+1}x_{2i+2}/x_{2i-1}x_{2i}/\cdots x_{2k})$\\ for $1 \leq i \leq k-1$.
\end{enumerate}

\begin{theorem}\label{vs_iso}
    The quotient module $\Q\left\langle \Perm_I \right\rangle / M_I$ is isomorphic to $\widetilde{H}_{k-1}((K_{A_{n}})_{I})$ as a $\Q$-vector space.
\end{theorem}
\begin{proof}
    We consider the map $\Q\left\langle \Perm_I \right\rangle / M_I \to \widetilde{H}_{k-1}((K_{A_{n}})_{I})$ induced by $\Phi_I$.
    By Remark~\ref{rem:relations}, it is well-defined, and, by Lemma~\ref{lemma:basis_A}, it is surjective.
    Hence, it is enough to show that every permutation can be expressed as a linear combination of alternating permutation by the relations.

    By the first relation, we may assume that each permutation
    $$
    y = (y_1y_2/\cdots/y_{2k-1}y_{2k}) \in \Q\left\langle \Perm_I \right\rangle / M_I
    $$
    satisfies $y_{2i-1} < y_{2i}$ for all $1 \leq i \leq k$.
    For any given $$x = (x_1\cdots/x_{2i-1}x_{x_{2i}}/x_{2i+1}x_{2i+2}/\cdots x_{2k}) \in \Perm_I,$$ if $x_{2i} < x_{2i+1}$, then $(x_{2i-1}x_{2i}/x_{2i+1}x_{2i+2})$ is the minimal element of $$(\Perm_{\{x_{2i-1},x_{2i},x_{2i+1},x_{2i+2}\}},\leq_R).$$
    Thus, by taking the second relation for $\{x_{2i-1},x_{2i},x_{2i+1},x_{2i+2}\}$, we observe that $x$ can be expressed by a linear sum of $y$s satisfying $x \leq_R y$.
    We do this procedure recursively, and it has to terminate in finite steps, which proves the theorem.

\end{proof}

Thus, we focus $\Q$-vector space isomorphisms
$$\Q\left\langle \Alt_I \right\rangle \overset{\iota}{\longrightarrow}  \frac{\Q\left\langle \Perm_I  \right\rangle}{M_I} \overset{\bar{\Phi}_I}{\longrightarrow} \widetilde{H}_{k-1}((K_{A_{n}})_{I}) \overset{\ast}{\longrightarrow} \widetilde{H}^{k-1}((K_{A_{n}})_{I}),$$
where
\begin{itemize}
  \item $\iota \colon \alpha \mapsto \alpha + M_I$,
  \item $M_I$ is the submodule defined immediately preceding Theorem~\ref{vs_iso},
  \item $\bar{\Phi}_I$ is induced by $\Phi_I \colon \Q\left\langle \Perm_I  \right\rangle \to \widetilde{H}_{k-1}((K_{A_n})_I)$, and
  \item $\ast$ is the canonical isomorphism as the dual space.
\end{itemize}
In conclusion, every element $x \in \Perm_I$ can be expressed as a linear combination of elements from $\Alt_I$ in $\Q\left\langle \Perm_I  \right\rangle/ M_I$.
Then, we denote by $\cC^{\alpha}_x$ the coefficient of $\alpha \in \Alt_I$ in the expression of~$x$, as
$$
x = \sum_{\alpha \in \Alt_I}\cC^{\alpha}_x\cdot\alpha.
$$

\begin{example}
Let $I = \{1,3,4,5\} \in \binom{[8]}{4}$.
Refer Example~\ref{A7} for considering $\widehat{(K_{A_7})}_I$.
Then, $(14/35) + (41/35)= 0$ in $\Q\left\langle \Perm_I  \right\rangle/ M_I$, which confirms that
$$
\cC^{(14/35)}_{(41/35)} = -1.
$$
Moreover, one can observe that
$$
(13/45) = (14/35)-(15/34) - (34/15) + (35/14) - (45/13).
$$
Then, we have $\cC^{(14/35)}_{(13/45)} = \cC^{(35/14)}_{(13/45)} = 1$, and $\cC^{(15/34)}_{(13/45)} =\cC^{(34/15)}_{(13/45)} =\cC^{(45/13)}_{(13/45)} = -1$. 
\end{example}

Furthermore, we extend
$$
    \bigoplus_{I \in \binom{[n+1]}{2k}}\Q\left\langle \Alt_I \right\rangle  \longrightarrow \bigoplus_{I \in \binom{[n+1]}{2k}} \frac{\Q\left\langle \Perm_I  \right\rangle}{M_I} \longrightarrow H_{k}(X^{\R}_{A_n})  \longrightarrow H^{k}(X^{\R}_{A_n}),
$$
and, we have a $\Q$-vector space isomorphism
\begin{equation}\label{isomorphism}
  \boldsymbol{\Phi} \colon
    \bigoplus_{k=0}^{\lfloor \frac{n+1}{2} \rfloor} \bigoplus_{I \in \binom{[n+1]}{2k}} \frac{\Q\left\langle \Perm_I  \right\rangle}{M_I} \to H^\ast(X^{\R}_{A_n}).
\end{equation}
It is clear that the natural permutation action of $W_{A_n}$ on $[n+1]$ induces
a $W_{A_n}$-module structure on $\bigoplus_{I\in \binom{[n+1]}{2r}} \frac{\Q\left\langle \Perm_I  \right\rangle}{M_I}$ since $\bigoplus_{I\in \binom{[n+1]}{2r}} M_I$ is stable under the action.
As a corollary, this gives an explicit (co)homology representation on $W_{A_n}$-action for $X^{\R}_{A_n}$ that is isomorphic to Specht modules as provided in \cite{Cho-Choi-Kaji2019}.

\section{Multiplicative structure of $H^\ast(X^\R_{A_n})$}\label{section : multi}

In this section, our investigation centers on the multiplicative structure of $H^\ast(X^\R_{A_n})$, which we relate to alternating permutations on even subsets within the set $[n+1]$.
For the purpose of this analysis, we will make the assumption that $I_1 \in \binom{[n+1]}{2k}$ and $I_2 \in \binom{[n+1]}{2l}$.

Now, we assume $I_1 \cap I_2 = \emptyset$, and in this case, we denote by $I_1+I_2$ the symmetric difference $I_1 \triangle I_2$.
For an alternating permutation $z = (z_1z_2/\cdots/z_{2(r+l)-1}z_{2(r+l)})$ in $\Alt_{I_1 + I_2}$, $z$ is said to be \emph{restrictable} to $I_1$ and $I_2$, if $\{z_{2i-1},z_{2i}\}$ is contained in either $I_1$ or $I_2$ for all $1 \leq i \leq r+l$.
We denote by $\cR_{I_1,I_2}$ the set of restrictable permutations to $I_1$ and $I_2$.

For $J \subset I \subset [n+1]$, we define a map $\rho_J \colon \Perm_I \to \Perm_J$ by $\rho_J(x)$ is a subpermutation of $x$.
Let $z \in \cR_{I_1,I_2}$, and set $x = \rho_{I_1} (z) = (x_1x_2/\cdots/x_{2k-1}x_{2k})$ and $y=\rho_{I_2}(z) = (y_1y_2/ \cdots / y_{2l-1}y_{2l})$.
Since both $x$ and $y$ are subpermutations of $z$, there is a rearrangement $\sigma \in \fS_{k+l}$ of indexes such that $(z_{2\sigma(1)-1}z_{2\sigma(1)}/\cdots/z_{2\sigma(k+l)-1}z_{2\sigma(k+l)})$ is
$(x_1x_2/\cdots/x_{2k-1}x_{2k}/y_1y_2/\cdots/y_{2l-1}y_{2l})$.
We call $\sigma$ the \emph{rearrangement} of $z$ with respect to $I_1$ and $I_2$.
Note that
$\sigma(1) < \sigma(2) < \cdots < \sigma(k)$ and $\sigma(l+1) < \sigma(l+2) < \cdots < \sigma(l+k)$.
The sign of $\sigma$ is denoted by $\kappa_{I_1,I_2}(z)$.

\begin{example}
  For $I_1 = \{1,3,4,6\}$ and $I_2 = \{2,5\}$, $z = (14/25/36)$ is a restrictable permutation to $I_1$ and $I_2$, \ie $z \in \cR_{I_1,I_2}$.
  Moreover, $\rho_{I_1}((14/25/36)) = (14/36)$, $\rho_{I_2}((14/25/36)) = (25)$, and $\kappa_{I_1,I_2}(z) = -1.$
  However, for example, $(15/34/26)$ is not restrictable to $I_1$ and $I_2$.
\end{example}

\begin{definition}\label{cup_perm}
Let $\smile \colon \frac{\Q\left\langle \Perm_{I_1}  \right\rangle}{M_{I_1}} \otimes \frac{\Q\left\langle \Perm_{I_2}  \right\rangle}{M_{I_2}} \to \frac{\Q\left\langle \Perm_{I_1 \triangle I_2}  \right\rangle}{M_{I_1 \triangle I_2}}$ defined by, for $\alpha \in \Alt_{I_1}$ and $\beta \in \Alt_{I_2}$,
$$
\alpha \smile \beta = \begin{cases}
\underset{z \in \cR_{I_1,I_2}}{\sum}(-1)^{\kappa_{I_1,I_2}(z)} \cC^\alpha_{\rho_{I_1}(z)} \cC^\beta_{\rho_{I_2}(z)} \cdot z, & \mbox{if } I_1 \cap I_2 = \emptyset,\\
0, & \mbox{otherwise.}
\end{cases}
$$
\end{definition}

Then, from the multiplicative structure $\smile$,
$$
 \bigoplus_{k=0}^{\lfloor \frac{n+1}{2} \rfloor} \bigoplus_{I \in \binom{[n+1]}{2k}} \frac{\Q\left\langle \Perm_I  \right\rangle}{M_I}
$$
can be regarded as a $\Q$-algebra.

\begin{example}
Let $I_1 = \{1,3,4,6\}$ and $I_2 = \{2,5\}$.
Then, all elements $z$ of $\cR_{I_1,I_2}$ are grouped into elements with the same $\rho_{I_1}(z)$ and $\rho_{I_2}(z)$ as follows:
\begin{enumerate}
  \item $z= (14/36/25),(14/25/36),(25/14/36),$ when $\rho_{I_1}(z) = (14/36)$ and $\rho_{I_2}(z) = (25)$
  \item $z= (16/34/25),(16/25/34),(25/16/34),$ when $\rho_{I_1}(z) = (16/34)$ and $\rho_{I_2}(z) = (25)$
  \item $z= (34/16/25),(34/25/16),(25/34/16),$
  when $\rho_{I_1}(z) = (16/34)$ and $\rho_{I_2}(z) = (25)$
  \item $z= (36/14/25),(36/25/14),(25/36/14),$
  when $\rho_{I_1}(z) = (36/14)$ and $\rho_{I_2}(z) = (25)$
  \item $z= (46/13/25),(46/25/13),(25/46/13),$
  when $\rho_{I_1}(z) = (46/13)$ and $\rho_{I_2}(z) = (25)$
  \item $z = (13/25/46),$ when $\rho_{I_1}(z) = (13/46)$ and $\rho_{I_2}(z) = (25)$
\end{enumerate}
Consider $\alpha = (14/36) \in \Alt_{I_1}$ and $\beta = (25) \in \Alt_{I_2}$.
For the first five cases,
$$
\cC^\alpha_{\rho_{I_1}(z)} \cdot \cC^\beta_{\rho_{I_2}(z)} = \delta_{\alpha,\rho_{I_1}(z)}\cdot\delta_{\beta,\rho_{I_2}(z)}=
\begin{cases}
1, & \mbox{in the first case,} \\
0, & \mbox{otherwise.}
\end{cases}$$
In the last case, since $(13/46) = (14/36)-(16/34)-(34/16)+(36/14)-(46/13)$,
$$
\cC^\alpha_{\rho_{I_1}(z)} \cdot \cC^\beta_{\rho_{I_2}(z)} = \cC^{(14/36)}_{(13/46)} \cdot \delta_{(25),(25)} = 1.
$$
Finally, we conclude that $$(14/36) \smile (25) = (14/36/25) - (14/25/36) + (25/14/36) - (13/25/46).$$
\end{example}

Now, let us consider the $\Q$-linear map $\boldsymbol{\Phi}$ in (\ref{isomorphism}).

\begin{theorem}\label{main_A}
The $\Q$-linear map
$$
\boldsymbol{\Phi} \colon \bigoplus_{k=0}^{\lfloor \frac{n+1}{2} \rfloor} \bigoplus_{I \in \binom{[n+1]}{2k}} \frac{\Q\left\langle \Perm_I  \right\rangle}{M_I} \to \underset{k \in \N}\bigoplus H^k(X_{A_n}^{\R},\Q),
$$
is a $\Q$-algebra isomorphism.
\end{theorem}

First of all, we introduce a key lemma for the proof of Theorem~\ref{main_A}.
We assume that $I_1 \cap I_2 = \emptyset$.

\begin{lemma}\label{lemma2}
Let $\iota_{\ast} \colon \widetilde{H}_{k+l-1}((K_{A_{n}})_{I_1+I_2}) \to \widetilde{H}_{k+l-1}((K_{A_{n}})_{I_1} \star (K_{A_{n}})_{I_2})$ be the homomorphism induced by the simplicial maps $\iota \colon (K_{A_{n}})_{I_1+I_2} \hookrightarrow (K_{A_{n}})_{I_1} \star (K_{A_{n}})_{I_2}$.

$$
\iota_{\ast}(\Phi_{I_1+I_2}(z)) =
\begin{cases}
(-1)^{\kappa_{I_1,I_2}(z)}\Phi_{I_1}(\rho_{I_1}(z))\star\Phi_{I_2}(\rho_{I_2}(z)), & \mbox{for }z \in \cR_{I_1,I_2} \\
0, & \mbox{for }z \notin \cR_{I_1,I_2}
\end{cases}
$$
\end{lemma}

\begin{proof}
We bifurcate our proof into two distinct cases; one where an alternating permutation $z$ is an element of $\cR_{I_1,I_2}$, and the other where it is not.

To begin with, let us consider the case where $z \in \cR_{I_1,I_2}$.
Let $\sigma \in \fS_{k+l}$ be the rearrangement of $z$ with respect to $I_1$ and $I_2$.
Set $x = \rho_{I_1} (z) = (x_1x_2/\cdots/x_{2k-1}x_{2k})$ and $y=\rho_{I_2}(z) = (y_1y_2/ \cdots / y_{2l-1}y_{2l})$.
We have
\begin{align*}
    \Phi_{I_1+I_2}(z) &= [\{\set^z_1,\set^z_2\}\star\cdots\star\{\set^z_{2k+2l-1},\set^z_{2k+2l}\}] \\
    &= (-1)^{\kappa_{I_1,I_2}(z)}[\{\set^z_{2\sigma(1)-1},\set^z_{2\sigma(1)}\}\star\cdots\star\{\set^z_{2\sigma(k+l)-1},\set^z_{2\sigma(k+l)}\}].
\end{align*}
For $1 \leq i \leq k$, we have $z_{2\sigma(i)-1} = x_{2i-1}$, $z_{2\sigma(i)} = x_{2i}$, and for $1 \leq j \leq l$, we have $z_{2\sigma(k+j)-1} = y_{2j-1}$ and $z_{2\sigma(k+j)} = y_{2j}$.
Since $\sigma(1) < \cdots < \sigma(k)$ and $\sigma(l+1) < \cdots < \sigma(l+k)$,
\begin{enumerate}
  \item $\set^z_{2\sigma(i)-1} \cap I_1 = \set^x_{2i-1}$ and $\set^z_{2\sigma(i)} \cap I_1 = \set^x_{2i}$
  \item $\set^z_{2\sigma(k+j)-1} \cap I_2 = \set^y_{2j-1}$ and $\set^z_{2\sigma(k+j)} \cap I_2 = \set^y_{2j}$.
\end{enumerate}
Then, by Lemma~\ref{typeA_homotopy},
\begin{align*}
    \Phi_{I_1}(\rho_{I_1}(z)) &=[\{\set^x_1,\set^x_2\}\star\cdots\star\{\set^x_{2k-1},\set^x_{2k}\}] \\ &=[\{\set^z_{2\sigma(1)-1},\set^z_{2\sigma(1)}\}\star\cdots\star\{\set^z_{2\sigma(k)-1},\set^z_{2\sigma(k)}\}]
\end{align*}
in $\widetilde{H}_{k-1}((K_{A_{n}})_{I_1})$, and
\begin{align*}
    \Phi_{I_2}(\rho_{I_2}(z)) &=[\{\set^y_1,\set^y_2\}\star\cdots\star\{\set^y_{2l-1},\set^y_{2l}\}]\\
    &=[\{\set^z_{2\sigma(k+1)-1},\set^z_{2\sigma(k+1)}\}\star\cdots\star\{\set^z_{2\sigma(k+l)-1},\set^z_{2\sigma(k+l)}\}]
\end{align*}
in $\widetilde{H}_{l-1}((K_{A_{n}})_{I_2})$.
Thus,
$[\Delta_z] = [\Delta_{\rho_{I_1}(z)}\star\Delta_{\rho_{I_2}(z)}]$ in $\widetilde{H}_{k+l-1}((K_{A_{n}})_{I_1}\star(K_{A_{n}})_{I_2})$.

Subsequently, we turn our consideration to the case where $z \notin \cR_{I_1,I_2}$.
We can choose $1 \leq q \leq k+l$ such that $\{z_{2q-1},z_{2q}\} \not\subset I_1$ and $\{z_{2q-1},z_{2q}\} \not\subset I_2$.
Then, considering $\set^z_{2q-1}$ and $\set^z_{2q}$, one of them is a vertex of $(K_{A_{n}})_{I_1}$, and the other one is a vertex of $(K_{A_{n}})_{I_2}$.
Therefore, the simplex $\{\set^z_{2q-1}, \set^z_{2q}\}$ is a $1$-face in $(K_{A_{n}})_{I_1}\star(K_{A_{n}})_{I_2}$.

Define the permutation $\hat{z}$ by removing $(z_{2q-1}z_{2q})$ from $z$.
Then, $\hat{z}$ is still a permutation, and $\Delta_{\hat{z}}$ is a nonzero element of $C_{k+l-2}((K_{A_{n}})_{I_1}\star(K_{A_{n}})_{I_2})$.
Denote by $\mu_z$ the element of $C_{k+l-1}((K_{A_{n}})_{I_1}\star(K_{A_{n}})_{I_2})$ obtained by joining $\Delta_{\hat{z}}$ and the $1$-simplex with two vertices  $\set^z_{2q-1}$ and $\set^z_{2q}$.
Then, $\partial(\mu_z)$ is either $\Delta_z$ or $-\Delta_z$, and, hence, $\Phi_{I_1+I_2}(z) = [\Delta_z] =[\pm \partial(\mu_z)] = 0$.

\end{proof}

From Theorem~\ref{Choi-Park2017 multimplicative}, the product structure on $\underset{I}\bigoplus \widetilde{H}^{\ast}((K_{A_{n}})_I)$ is given by the canonical maps $$\smile \colon \widetilde{H}^{k-1}((K_{A_{n}})_{I_1}) \otimes \widetilde{H}^{l-1}((K_{A_{n}})_{I_2}) \to \widetilde{H}^{k+l-1}((K_{A_n})_{I_1+I_2})$$ which are induced by simplicial maps $\iota \colon (K_{A_{n}})_{I_1+I_2} \hookrightarrow (K_{A_{n}})_{I_1} \star (K_{A_{n}})_{I_2}$.
In other words, the product $g \smile h$ of $g \in \widetilde{H}^{k-1}((K_{A_n})_{I_1})$ and $h \in \widetilde{H}^{l-1}((K_{A_n})_{I_2})$ is $(g \otimes h) \circ \iota_{\ast}$, as the following commutative diagram ;

\begin{eqnarray}\label{dia}
  \begin{tikzcd}[row sep = huge]
    \widetilde{H}_{k+l-1}((K_{A_n})_{I_1+ I_2}) \arrow{r}{\iota_{\ast}} \arrow[swap]{dr}{g \smile h} & \widetilde{H}_{k+l-1}((K_{A_n})_{I_1} \star (K_{A_n})_{I_2}) \arrow{d}{g \otimes h} \\
     & \Q
  \end{tikzcd},
\end{eqnarray}

where $\iota_{\ast}$ is the induced homomorphism of $\iota$.

\begin{proof}[Proof of Theorem~\ref{main_A}.]
It is sufficient to prove that $\boldsymbol{\Phi}$ preserves multiplicative structures by demonstrating that
$$
(\boldsymbol{\Phi}(\alpha) \smile \boldsymbol{\Phi}(\beta))(\Phi_{I_1 \triangle I_2}(z)) = \boldsymbol{\Phi}(\alpha \smile \beta)(\Phi_{I_1 \triangle I_2}(z))
$$
for each $\alpha \in \Alt_{I_1},$
$\beta \in \Alt_{I_2}$, and $z \in \Perm_{I_1 \triangle I_2}.$

If $I_1 \cap I_2 \neq \emptyset$, then $\vert I_1 \triangle I_2 \vert = \vert I_1 \cup I_2 \vert - \vert I_1 \cap I_2 \vert < 2k+2l$.
Then, since $(K_{A_{n}})_{I_1 \triangle I_2}$ is homotopy equivalent to a wedge sum of sphere which dimension is at most $k+l-2$, $\widetilde{H}^{k+l-1}((K_{A_{n}})_{I_1 \triangle I_2})$ is trivial.
Thus, $\widetilde{H}^{k-1}((K_{A_{n}})_{I_1}) \otimes \widetilde{H}^{l-1}((K_{A_{n}})_{I_2}) \to \widetilde{H}^{k+l-1}((K_{A_{n}})_{I_1 \triangle I_2})$ is a zero map.
In other words, the multiplicative structure of $\widetilde{H}^{k-1}((K_{A_{n}})_{I_1})$ and $\widetilde{H}^{l-1}((K_{A_{n}})_{I_2})$ is trivial, and thus $(\boldsymbol{\Phi}(\alpha) \smile \boldsymbol{\Phi}(\beta)) =0$.
Since $\alpha \smile \beta$ is zero, the proof is complete, if $I_1 \cap I_2 \neq \emptyset$.

Hence, we assume that $I_1 \cap I_2 = \emptyset$.
From \eqref{dia}, $$(\boldsymbol{\Phi}(\alpha) \smile \boldsymbol{\Phi}(\beta))(\Phi_{I_1 + I_2}(z)) = (\boldsymbol{\Phi}(\alpha) \otimes \boldsymbol{\Phi}(\beta))(\iota_{\ast}(\Phi_{I_1 + I_2}(z))).$$
By Lemma~\ref{lemma2},
$$
\iota_{\ast}(\Phi_{I_1+I_2}(z)) =
\begin{cases}
(-1)^{\kappa_{I_1,I_2}(z)}\Phi_{I_1}(\rho_{I_1}(z))\star\Phi_{I_2}(\rho_{I_2}(z)), & \mbox{for }z \in \cR_{I_1,I_2}, \\
0, & \mbox{for }z \notin \cR_{I_1,I_2}.
\end{cases}
$$
Thus, if $z \in \cR_{I_1,I_2}$,
\begin{align*}
  (\boldsymbol{\Phi}(\alpha) \smile \boldsymbol{\Phi}(\beta))(\Phi_{I_1 + I_2}(z)) &= (-1)^{\kappa_{I_1,I_2}(z)}\boldsymbol{\Phi}(\alpha)(\Phi_{I_1}(\rho_{I_1}(z))) \boldsymbol{\Phi}(\beta)(\Phi_{I_2}(\rho_{I_2}(z)))\\
  &= (-1)^{\kappa_{I_1,I_2}(z)} \cC^\alpha_{\rho_{I_1}(z)}\cC^\beta_{\rho_{I_2}(z)}.
\end{align*}
If $z \notin \cR_{I_1,I_2}$, $(\boldsymbol{\Phi}(\alpha) \smile \boldsymbol{\Phi}(\beta))(\Phi_{I_1 + I_2}(z)) = 0.$

Now, by Definition~\ref{cup_perm},
\begin{align*}
    \boldsymbol{\Phi}(\alpha \smile \beta)(\Phi_{I_1 + I_2}(z)) &= \boldsymbol{\Phi}(\sum_{\bar{z} \in \cR_{I_1,I_2}}(-1)^{\kappa_{I_1,I_2}(\bar{z})}\cC^\alpha_{\rho_{I_1}(\bar{z})}\cC^\beta_{\rho_{I_2}(\bar{z})} \cdot \bar{z})(\Phi_{I_1 + I_2}(z)) \\
    &= \begin{cases}
(-1)^{\kappa_{I_1,I_2}(z)}\cC^\alpha_{\rho_{I_1}(z)}\cC^\beta_{\rho_{I_2}(z)}, & \mbox{for }z \in \cR_{I_1,I_2}, \\
0, & \mbox{for }z \notin \cR_{I_1,I_2},
\end{cases}
\end{align*}
as desired.
\end{proof}

\section{Further discussions on coefficients}\label{section : Coefficients}

In this section, for any given $\alpha \in \Alt_I$ and $x \in \Perm_I$, we shall discuss the computation of the coefficient $\cC^{\alpha}_x$ using \emph{incidence algebras}.
For a finite poset $(\cP,\leq)$, consider the set of functions from intervals of $\cP$ to integers, denoted by $\cI(\cP,\Z)$.
Define addition and scalar multiplication on $\cI(\cP,\Z)$ pointwise.
The multiplication $\ast$ on $\cI(\cP,\Z)$ is defined by
$$
f \ast g ([a,b]) = \sum_{a \leq x \leq b}f([a,x])g([x,b])
$$
for each pair $f$ and $g$ of $\cI(\cP,\Z)$.
It is well-known that $\cI(\cP,\Z)$ is an associative $\Z$-algebra under~$(+,\ast,\cdot_{\Z})$, and is called the incidence algebra on $(\cP,\leq)$ over $\Z$.

For each $I \in {[n+1] \choose 2k}$, we define a relation $\lessdot$ such that $x \lessdot \alpha$ if and only if
\begin{enumerate}
  \item $\varphi_I(\alpha) = \{\set_1^\alpha,\set_3^\alpha,\ldots,\set_{2k-1}^\alpha\}$ is a face of $\Delta_x$, and
  \item $\alpha \in \Alt_I$
\end{enumerate}
for each pair of permutations $x \in \Perm_I$ and $\alpha \in \Alt_I$.
We now define a partial order $\prec$ on $\Perm_I$ by the covering relation $\lessdot$, \ie $x \prec \alpha$ if and only if there exist $\alpha_1,\ldots,\alpha_{k-1} \in \Perm_I$ such that
$$
x\lessdot \alpha_1 \lessdot \cdots \lessdot \alpha_{k} = \alpha.
$$

By Lemma~\ref{face_lemma}, if $x \lessdot \alpha$, then there uniquely exist $\tau_{\alpha,x} \in \fS_{2k}^{\even}$ and $\sigma_{\alpha,x} \in \fS_{2k}^{\odd}$ such that $$
\tau_{\alpha,x}\cdot \sigma_{\alpha,x} \cdot \alpha = x.
$$
Let $f_I \in \cI(\Perm_I,\prec)$ be defined inductively by the following relation:
$$
f_I(x,\alpha) = \begin{cases}
        -1, & \mbox{if } \alpha = x \notin \Alt_I,\\
        \delta_{\alpha,x}, & \mbox{if } x \in \Alt_I, \\
        -\underset{x \prec \beta \lessdot \alpha}{\sum}\sgn(\tau_{\alpha,\beta})f_{I}(x,\beta), & \mbox{otherwise}.
      \end{cases}
$$

\begin{example}
For a set $I = \{1,2,3,4\}$, the induced subposet of $(\Perm_I,\prec)$ spanned by the set~$\{\alpha \in \Perm_I \mid x \prec \alpha\}$ is
\begin{center}
\begin{tikzpicture}[->,>=stealth',shorten >=1pt,auto,node distance=1.8cm,
                    thick,main node/.style={rectangle,draw,font=\sffamily\small}]

  \node [main node] (1) {34/12};
  \node [main node](2) [below of=1] {13/24};
  \node [main node](3) [right of=2] {14/23};
  \node [main node](4) [right of=3] {23/14};
  \node [main node](5) [left of=2] {24/13};
  \node [main node](6) [below of=2] {12/34};

  \path[every node/.style={font=\sffamily\small}]
    (6) edge node [right] {} (2)
    (6) edge node [right] {} (3)
    (6) edge node [right] {} (4)
    (6) edge node [right] {} (5)
    (2) edge node [right] {} (1);
\end{tikzpicture}.
\end{center}
Let $x$ denote $(12/34)$, and then, $f_I([x,x]) = -1$. For each $\alpha \in \Alt_I$, the value $f_I([x,\alpha])$ is
$$ -\underset{x \prec \beta \lessdot \alpha}{\sum}\sgn(\tau_{\alpha,\beta})f_I([x,\beta]) = -\sgn(\tau_{\alpha,x})f_I([x,x]) =\sgn(\tau_{\alpha,x}) =\begin{cases}
1, & \mbox{if } \alpha = (13/24),\\
-1, & \mbox{if } \alpha = (14/23),\\
-1, & \mbox{if } \alpha = (23/14),\\
1, & \mbox{if } \alpha = (24/13),
\end{cases}$$
and, if $\alpha = (34/12)$,
$$
f_I([x,\alpha]) = -\sgn(\tau_{\alpha,(13/24)})f_I([x,(13/24)]) = -1.
$$
In this case, one can observe that $\cC^\alpha_{x} = f_I([x,\alpha])$ for all $\alpha \in \Alt_I.$
\end{example}

The following theorem demonstrates that this phenomenon is not merely a coincidence in specific cases, but rather a fact that holds in general.

\begin{theorem}
Let $x \in \Perm_I$.
For each $\alpha \in \Alt_I$,
$$
\cC^{\alpha}_x = \begin{cases}
  f_I([x,\alpha]), & \mbox{if } x \prec \alpha, \\
  0, & \mbox{otherwise}.
\end{cases} 
$$
\end{theorem}

\begin{proof}
    Let a permutation $x \in \Perm_I$ be given. 
    If $x$ is alternating, then $$
    \cC^\alpha_x = \delta_{x,\alpha} = f_I([x,\alpha])
    $$
    for each $\alpha \in \Alt_I$.
    Now, we assume that $x \in \Perm_I$ is not an alternating permutation on $I$.
    To establish the proof, let us consider two cases; when the alternating permutations $\alpha \in \Alt_I$ satisfies $x \prec \alpha$, or not.
    
\medskip
  \noindent 
    \underline{\textbf{CASE 1. $x \prec \alpha$ :}}
    Let $\alpha$ be a minimal element of $(\Alt_I, \prec \vert_{\Alt_I})$, \ie $x \lessdot \alpha$.
    The coefficient of $\varphi_I(\alpha)$ in the expression~$\Phi_I(x)$ is $\sgn(\tau_{\alpha,x})$.
    Then,
    $$
    \cC^\alpha_x = \sgn(\tau_{\alpha,x}) = -\sgn(\tau_{\alpha,x})f_I([x,x]) = f_I([x,\alpha]).
    $$
    We assume that, for each $\beta \in \Alt_I$, $\cC^\beta_x = f_I([x,\beta])$ if $\beta \prec \alpha$.
    Thus,
    \begin{align*}
    \cC^{\alpha}_x &=\sum_{x \prec \beta \lessdot \alpha}{-\sgn(\tau_{\alpha,\beta})\cC^{\beta}_x}\\
    &= -\sum_{x \prec \beta \lessdot \alpha}{\sgn(\tau_{\alpha,\beta})f_I([\beta,x])}\\
    &= \cC^{\alpha}_x.
\end{align*}
Therefore, we conclude that this theorem holds in this case, inductively.
    
    \medskip
    \noindent
    \underline{\textbf{CASE 2. $x \not\prec \alpha$ :}}
    It can be observed that $\varphi_I(\alpha)$ is neither a face of $\Delta_x$ nor $\Delta_\beta$ for all $\beta \prec \alpha$, where $\beta \in \Alt_I$.
    Then, we immediately obtain that $\cC^\alpha_x = 0$ as desired.
\end{proof}

\begin{example}\label{graph_ex}
For $I = \{1,2,3,4,5,6\}$, the induced subposet of $(\Perm_I,\prec)$ spanned by the interval $[(12/34/56),(56/34/12)]$ is
\begin{center}
\begin{tikzpicture}[->,>=stealth',shorten >=1pt,auto,node distance=2cm,
                    thick,main node/.style={rectangle,draw,font=\sffamily\footnotesize}]
  \node [main node] (0) {35/16/24};
  \node [main node] (3) [below left of =0]{13/25/46};
  \node [main node, rounded corners=8pt] (4) [below right of=3] {12/34/56};
  \node [main node] (5) [left of=0] {35/14/26};
  \node [main node] (6) [right of=0] {15/34/26};
  \node [main node] (11)[above right of=5] {35/46/12};
  \node [main node] (12)[above right of=0] {56/13/24};
  \node [main node] (15)[above right of=11]{56/34/12};

  \path[every node/.style={font=\sffamily\small}]
    (3) edge node [right] {} (0)
    (4) edge node [right] {} (3)
    (3) edge[left=32] node [left] {} (5)
    (3) edge node [below right] {} (6)
    (5) edge node [right] {} (11)
    (6) edge node [right] {} (12)
    (0) edge node [below right, yshift=-3mm] {} (15)
    (11) edge node [left] {} (15)
    (12) edge node [right] {} (15);
\end{tikzpicture}.
\end{center}
Let $x$ denote $(12/34/56)$. For some $\alpha \in \Alt_I$, we obtain that the coefficients $\cC^\alpha_x$ as follows:
\begin{enumerate}
  \item
  For $\alpha = (13/25/46)$,
  $$
  x = \tau_{\alpha,x}\cdot\sigma_{\alpha,x}\cdot \alpha = \tau_{\alpha,x}\cdot (12/34/56) = (12/34/56),
  $$
  and then,
  $\cC^{\alpha}_{x} = f_I([x,\alpha]) = -\sgn(\tau_{\alpha,x})f_I([x,x])= 1$.
  \item For $\alpha = (35/14/26)$ and $\beta = (13/25/46)$,
  $$
  \beta = \tau_{\alpha,\beta}\cdot\sigma_{\alpha,\beta}\cdot \alpha = \tau_{\alpha,\beta}\cdot (31/52/46) = (13/25/56),
  $$
  and then,
  $\cC^{\alpha}_{x} = f_I([x,\alpha]) = -\sgn(\tau_{\alpha,\beta})f_I([x,\beta])= -1$.
  By similar way,
  $$
  \cC^{(35/16/25)}_x = 1 \text{ and } \cC^{(15/34/26)}_x = -1.
  $$
  \item For $\alpha = (35/46/12)$ and $\beta = (35/14/26)$,
  $$
  \beta = \tau_{\alpha,\beta}\cdot\sigma_{\alpha,\beta}\cdot \alpha = \tau_{\alpha,\beta}\cdot (35/41/62) = (35/14/26),
  $$
  and then,
  $\cC^{\alpha}_{x} = f_I([x,\alpha]) = -\sgn(\tau_{\alpha,\beta})f_I([x,\beta])= 1$.
  By similar way, $\cC^{(56/13/24)}_x = 1.$
  \item Let $\alpha = (56/34/12)$, $\beta_1 = (35/46/12)$, $\beta_2 = (35/16/24)$, and $\beta_3 = (56/13/24)$.
      Then,
      \begin{align*}
        \cC^\alpha_x & = f_I([\alpha,x]) \\
         &= -f_I([x,\beta_1])+f_I([x,\beta_2])-f_I([x,\beta_3])\\
         &= -1.
      \end{align*}
\end{enumerate}

\end{example}

Asking which elements of an incidence algebras on a poset take values in $\{-1,0,1\}$ is an important question.
See \cite{Deodhar1977} and \cite{Bjorner1993} for references.
After conducting calculations for small positive integers~$k$, the authors were able to confirm that every non-zero $f_I([x,\alpha])$ is $\pm 1$, where~$I \in \binom{[n+1]}{2k}$ and~$x \prec \alpha \in \Perm_I$.
However, it is unclear whether $f_I([x,\alpha])$ always takes $\{-1, 0, 1\}$.
This leads us to propose the following conjecture.
\begin{conjecture}
For each interval $[x,\alpha]$ in $(\Perm_I,\prec)$, $f_I([x,\alpha])$ takes values in $\{-1,0,1\}$.
\end{conjecture}

\end{document}